\documentclass[12pt, reqno]{amsart}

\usepackage{amsmath,amssymb,amscd,amsxtra,upref,stmaryrd}
\usepackage{latexsym}
\usepackage[dvips]{graphics,epsfig}
\usepackage{enumerate,mathrsfs,comment,mathtools}
\usepackage[colorlinks=true, pdfstartview=FitV, linkcolor=blue, citecolor=blue, urlcolor=blue]{hyperref}

\newtheorem{theorem}{Theorem}[section]

\newtheorem{corollary}[theorem]{Corollary}

\theoremstyle{remark}
\newtheorem{definition}{Definition}[section]

\newtheorem{remark}{Remark}[section]

\def\Xint#1{\mathchoice
{\XXint\displaystyle\textstyle{#1}}%
{\XXint\textstyle\scriptstyle{#1}}%
{\XXint\scriptstyle\scriptscriptstyle{#1}}%
{\XXint\scriptscriptstyle\scriptscriptstyle{#1}}%
\!\int}
\def\XXint#1#2#3{{\setbox0=\hbox{$#1{#2#3}{\int}$}
\vcenter{\hbox{$#2#3$}}\kern-.5\wd0}}

\def\dashint{\Xint-}

\newcommand{\com}{\mathbb{C}}
\newcommand{\na}{\mathbb{N}}
\newcommand{\naz}{\mathbb{N}_0}

\newcommand{\re}{\mathbb{R}}
\newcommand{\rn}{{{\mathbb R}^n}}
\newcommand{\rtn}{\re^{2n}}

\newcommand{\abs}[1]{\left\vert #1 \right\vert}

\newcommand{\norm}[2]{\left\|#1\right\|_{#2}}

\newcommand{\sw}{{\mathcal{S}}(\rn)}
\newcommand{\swp}{{\mathcal{S}'}(\rn)}

\newcommand{\fhat}{\widehat{f}}
\newcommand{\ghat}{\widehat{g}}

\newcommand{\dx}{\, dx}
\newcommand{\dy}{\, dy}

\newcommand{\dxi}{\, d\xi}
\newcommand{\deta}{\, d\eta}

\newcommand{\ms}{\mathcal{M}^{\#}}
\newcommand{\hl}{\mathcal{M}}
\newcommand{\hld}{\mathcal{M}_2}

\newcommand{\qc}{{x_Q}}
\newcommand{\expo}{\varepsilon}
\newcommand{\qtil}{{\tilde{Q}}}
\newcommand{\rtil}{{\tilde{R}}}
\begin{document}

\title[Strongly singular bilinear operators] {Strongly singular bilinear Calder\'on-Zygmund operators and a class of bilinear pseudodifferential operators}
\author{\'Arp\'ad B\'enyi, Lucas Chaffee \and Virginia Naibo}

\address{\'Arp\'ad B\'enyi, Department of Mathematics,
516 High St, Western Washington University, Bellingham, WA 98225,
USA.} \email{arpad.benyi@wwu.edu}

\address{Lucas Chaffee, Department of Mathematics,
516 High St, Western Washington University, Bellingham, WA 98225,
USA.} \email{lucas.chaffee@wwu.edu}

\address{Virginia Naibo, Department of Mathematics, Kansas State University.
138 Cardwell Hall, 1228 N. 17th Street, Manhattan, KS  66506, USA.}
\email{vnaibo@ksu.edu}

\thanks{The first author is partially supported by a grant from the Simons Foundation (No. 246024). The third author is partially supported by the NSF under grant DMS 1500381.}

\subjclass[2010]{Primary: 35S05, 47G30, 42B20.  Secondary: 42B15}


\date{\today}

\keywords{Stronly singular bilinear operators, Calder\'on--Zygmund operators, bilinear pseudodifferential operators}

\begin{abstract}
Motivated by the study of kernels of bilinear pseudodifferential operators with symbols in a  H\"ormander class of critical order, we investigate  boundedness properties of strongly singular Calder\'on--Zygmund operators in the bilinear setting.  For such operators, whose kernels satisfy integral-type conditions, we establish boundedness properties in the setting of Lebesgue spaces as well as endpoint  mappings involving the space of functions of bounded mean oscillations and the Hardy space. Assuming pointwise-type conditions on the kernels, we show that strongly  singular bilinear Calder\'on--Zygmund operators satisfy pointwise estimates in terms of maximal operators, which imply their boundedness in weighted Lebesgue spaces.
\end{abstract}

\maketitle

\section{Introduction}

Bilinear pseudodifferential operators with symbols in the bilinear H\"ormander classes $BS_{\rho, \delta}^m$  are a priori defined from $\mathcal S(\rn)\times\mathcal S(\rn)$ into $\mathcal S'(\rn)$ and have the form
\[T_\sigma (f,g)(x):= \int_{\rtn} \sigma(x, \xi, \eta) \fhat(\xi) \ghat(\eta) e^{ix\cdot(\xi+\eta)} \dxi \deta \quad \forall x \in \rn,\, f,g\in\sw;
\]
given $0\leq \delta\leq \rho\leq 1$ and $m \in \re,$ the corresponding symbol $\sigma$ belongs to the class $BS_{\rho, \delta}^m$ if $\sigma:\rn\times\rn\times\rn\to\com$ is an infinitely differentiable function such that for any given multi-indices $\alpha,\beta,\gamma\in \naz$ there exists $C_{\alpha,\beta,\gamma}>0$ satisfying
\begin{equation}\label{def:Bmrd}
|\partial_x^\alpha \partial_\xi^\beta \partial_\eta^\gamma \sigma(x, \xi, \eta)| \le C_{\alpha,\beta,\gamma} (1+|\xi|+|\eta|)^{m +\delta \abs{\alpha}-\rho(\abs{\beta+\gamma})} \quad \forall x, \xi, \eta\in \rn.
\end{equation}
Such operators have been  studied extensively and their boundedness properties have been proved in a variety of settings; see  the articles
B\'enyi--Bernicot--Maldonado--Naibo--Torres~\cite{MR3205530}, B\'enyi--Maldonado--Naibo--Torres~\cite{MR2660466}, B\'enyi--Torres~\cite{MR1986065, MR2046194}, Brummer--Naibo~\cite{MR3750234}, Herbert--Naibo~\cite{MR3211086, MR3627725},  Koezuka--Tomita~\cite{MR3750316}, Michalowski--Rule--Staubach \cite{MR3165300},  Miyachi--Tomita \cite{MR3179688, MiTo, MiTo2}, Naibo \cite{MR3393696,MR3411149}, Naibo--Thomson~\cite{NaTho}, Rodr\'{\i}guez-L\'opez--Staubach~\cite{MR3035059} and references therein.

 The work in this manuscript is motivated by bilinear pseudodifferential operators with symbols in the class $BS^{-n(1-\rho)}_{\rho,\delta}$ with $0\leq \delta\leq \rho<1.$  As shown  in \cite[Theorem 2.2]{MR3205530}, if $\sigma\in BS^{m}_{\rho,\delta}$ with $m<-n(1-\rho),$  then $T_\sigma$ is bounded from $L^\infty(\rn)\times L^\infty(\rn)$ into $L^\infty(\rn);$ moreover, it was proved in \cite[Theorem A.2]{MR3179688} that if $\sigma\in BS^{m}_{\rho,\delta}$ with $m>-n(1-\rho),$ then $T_\sigma$ may fail to be bounded from $L^\infty(\rn)\times L^\infty(\rn)$ into $BMO(\rn).$ In view of this, $-n(1-\rho)$ and $BS^{-n(1-\rho)}_{\rho,\delta}$ are referred to as {\it a critical order} and {\it a critical bilinear H\"ormander class}, respectively.
 If $\sigma$ is in the critical class $BS_{\rho, \delta}^{-n(1-\rho)},$  then $T_\sigma$ is  also bounded from $L^\infty(\rn)\times L^\infty(\rn) $ into $BMO(\rn)$; this was first proved for  $0\le \rho< 1/2$ and $\delta=0$ in \cite[Theorem 2.4]{MR3205530} and then extended to $0\le \delta\le \rho< 1/2$  in \cite[Theorem 1.1]{MR3411149}.   Finally, the results in the recent manuscript \cite{MiTo}  settled the boundedness from $L^\infty(\rn)\times L^\infty(\rn) $ into $BMO(\rn)$ for all operators with symbols in the critical classes of order  $-n(1-\rho)$ with $0<\rho<1.$

In many instances, it is convenient to consider the kernel representation of  bilinear pseudodifferential operators with symbols in $BS^m_{\rho,\delta}$, formally
$$T_\sigma (f, g)(x)=\int_{\rtn} K_\sigma (x, y, z) f(y)g(z)\,dydz,$$
since the estimates of the symbol $\sigma$ translate into quantitative estimates on the kernel $K_\sigma,$ see \cite[Theorem E]{MR3205530}. For example, if $\sigma\in BS^0_{1,\delta}$ with $0\le \delta<1$ or  $\sigma\in BS^{m}_{\rho,\delta}$ with $m<-2n(1-\rho),$  $0\leq \delta\le \rho\le1,$  $\delta<1$ and $\rho>0$, then  $T_\sigma$ is a bilinear Calder\'on--Zygmund operator; that is, $T_\sigma$ is bounded from  $L^2(\rn)\times L^2(\rn)$ into $L^1(\rn)$ and the kernel $K_\sigma$ is a locally integrable function away from the diagonal $\Delta=\{(x,x,x): x\in\rn\}$ that, for some $\varepsilon>0,$ satisfies estimates of the form
\begin{align*}
|K_\sigma (x, y, z) | &\lesssim \big(|x-y| + |x- z|\big)^{-2n},
\end{align*}
\begin{align*}
&|K_\sigma (x+h, y, z) -K_\sigma(x,y,z)| +|K_\sigma (x, y+h, z) -K_\sigma(x,y,z)|+\\
&\hspace{3cm}|K_\sigma (x, y, z+h) -K_\sigma(x,y,z)| \lesssim |h|^\varepsilon\big(|x-y| + |x-z| \big)^{-2n-\varepsilon},
\end{align*}
for $(x,y,z)\in \mathbb R^{3n}\setminus \Delta $ and $|h|\le\frac{1}{2} (|x-y|+|x-z|).$ In particular,  the  bilinear Calder\'on--Zygmund theory established in Coifman--Meyer~\cite{MR518170} and Grafakos--Torres~\cite{ MR1880324}  gives that $T_\sigma$ is bounded from $L^{p}(\rn)\times L^{q}(\rn)$ into $L^r(\rn)$ with $1< p,q<\infty $ and $1/r=1/p+1/q$ , and from $L^\infty(\rn)\times L^\infty(\rn)$ into $BMO(\rn).$  However, except for some ranges of indices as those mentioned above, the classes $BS^m_{\rho,\delta}$ do not necessarily give rise to bilinear Calder\'on--Zygmund operators and  boundedness properties for the corresponding operators have been studied in the references cited at the beginning of this section.

In this article, we show that bilinear pseudodifferential operators with symbols in the critical classes $BS_{\rho, \delta}^{-n(1-\rho)}$ posses kernels that lie barely outside the scope of the bilinear Calder\'on--Zygmund theory. More specifically, the kernels of such operators are shown to be {\it strongly  singular bilinear Calder\'on--Zygmund kernels}, and when $0<\rho<1/2,$ the operators are shown to be {\it strongly  singular bilinear Calder\'on--Zygmund operators}. For the latter operators, whose kernels satisfy integral-type conditions, we  establish boundedness properties in the setting of Lebesgue spaces as well as the endpoint  mappings from $L^\infty(\rn)\times L^\infty(\rn)$ into $BMO(\rn),$ from $\mathcal{H}^1(\rn)\times L^\infty(\rn)$ into $L^1(\rn)$ and  from $ L^\infty(\rn)\times \mathcal{H}^1(\rn)$ into $L^1(\rn),$ where $\mathcal{H}^1(\rn)$ denotes the Hardy space. Assuming pointwise conditions on the kernels, we also show that strongly  singular bilinear Calder\'on--Zygmund operators satisfy pointwise estimates in terms of the sharp maximal operator and the Hardy--Littlewood maximal operator, which imply their boundedness in weighted Lebesgue spaces.

\bigskip

The article is organized as follows. In Section~\ref{sec:ssczo}, we define strongly singular bilinear Calder\'on--Zygmund operators and state and prove their boundedness properties in unweighted Lebesgue spaces as well as endpoint mappings involving $BMO(\rn)$ and $\mathcal{H}^1(\rn).$ Pointwise estimates   in terms of maximal operators and weighted estimates  in Lebesgue spaces for such singular operators are presented in Section~\ref{sec:maxweight}. Section~\ref{sec:critical}  discusses the realization of  bilinear pseudodifferential operators with symbols in the critical classes $BS^{-n(1-\rho)}_{\rho,\delta}$ as strongly singular bilinear Calder\'on--Zygmund operators.

We end this section with definitions and notation used throughout the manuscript.

\subsection*{Definitions and notation.}
The Hardy--Littlewood maximal operator will be denoted by $\hl$, that is,
$$
\hl(f)(x)=\sup_{x\in Q}\dashint_Q |f(y)|\,dy\quad \forall x\in\rn, f\in L^1_{\text{loc}}(\rn),
$$
where the supremum is taken over all cubes $Q\subset \rn$ containing $x$ and $\dashint_Q |f(y)|\dy=\frac{1}{|Q|}\int_Q |f(y)|\dy.$

A weight on $\rn$ is a non-negative locally integrable function defined on $\rn.$ Given $1<p<\infty,$ the Muckenhoupt class  $A_p$ is defined as the family of weights $w$ on $\rn$ such that
\[
\sup_{Q}\left(\dashint_Qw(y)\dy\right)\left(\dashint_Q w(y)^{1-p'}\dy\right)^{p-1}<\infty,
\]
where the supremum is taken over all cubes $Q\subset \rn$ and $p'$ is the conjugate exponent of $p.$ We set $A_\infty:=\cup_{p>1} A_p.$

Given a weight on $\rn$ and $0<p\le \infty,$ the notation $L^p_w(\rn)$ means the Lebesgue space with respect to the measure $w(x)\dx;$ if $w\equiv 1$ we simply write $L^p(\rn).$ We recall that for $1<p<\infty,$ $w\in A_p$ if and only if $\hl$ is bounded on $L^p_w(\rn).$

The sharp maximal operator $\ms$ is defined by
$$
\ms (f)(x)=\sup_{x\in Q}\dashint_Q |f(y)-f_Q|\,dy \quad \forall x\in\rn, f\in L^1_{\text{loc}}(\rn),
$$
where $f_Q=\dashint_Q f(y)\,dy$ and  the supremum is taken over  cubes $Q\subset \rn$ containing $x.$

The space of functions of bounded mean oscillations, $BMO(\rn),$  consists of all measurable functions defined on $\rn$ (identified modulo constants) such that
$$
\norm{f}{BMO}:=\norm{\ms(f)}{L^\infty}<\infty.
$$
We recall that the dual of  the Hardy space $\mathcal{H}^1(\rn)$ is $BMO(\rn).$

The Schwartz class of smooth rapidly decreasing functions on $\rn$ will be denoted by $\sw$ and its dual, the class of tempered distributions, by $\swp.$ The notation $C_c^\infty(\rn)$ will mean the space of compactly supported infinitely differentiable functions defined on $\rn.$ The space of bounded measurable functions defined on $\rn$ that have compact support will be indicated by $L^\infty_c(\rn).$

The Fourier transform of a tempered distribution $f$ is denoted by $\fhat.$

The notation $A\lesssim B$ means $A\le c \,B,$ where $c$ is a constant that may depend on some of the parameters and weights used but not on the functions involved.

\section{Strongly singular bilinear Calder\'on--Zygmund operators}\label{sec:ssczo}

In this section we first present the definitions of strongly singular bilinear Calder\'on--Zygmund kernels and operators, which are inspired by the work in the linear setting of Fefferman~\cite{MR0257819}, Fefferman--Stein~\cite{MR0447953} and Alvarez--Millman~\cite{MR849442}. We  then state and prove boundedness properties of these operators in the context  of Lebesgue spaces,  $\mathcal{H}^1(\rn)$ and $BMO(\rn).$

Let $T:\mathcal S(\rn)\times \mathcal S(\rn)\to \mathcal S'(\rn)$ be a continuous bilinear operator and $\mathcal{K}$ be a complex-valued locally integrable function  defined on $\re^{3n}\setminus \Delta.$ We say that $T$ is associated to $\mathcal{K}$ if for any $f, g\in C^\infty_c(\rn),$
\begin{equation}\label{eq:Tpointwise}
T(f, g)(x)=\int_{\rtn} {\mathcal K}(x, y, z)f(y)g(z)\,dydz\quad  \forall x\not\in \text{supp}\,(f)\cap\text{supp}\, (g).
\end{equation}
 The formal transposes of the operator $T$ will be denoted by $T^{*1}$ and $T^{*2}$ and are defined by
\[
\langle T(f,g), h\rangle=\langle T^{*1}(h,g), f\rangle=\langle T^{*2}(f,h), g\rangle\quad \forall f,g,h\in\sw.
\]
It follows that the kernels $\mathcal{K}^{*j}$ of $T^{*j}$  for $j=1,2$ are  given by
\[
\mathcal{K}^{*1}(x,y,z)=\mathcal{K}(y,x,z)\quad \text{ and }  \quad \mathcal{K}^{*2}(x,y,z)=\mathcal{K}(z,y,x).
\]

\begin{definition}
\label{def:strongly}
Let $T:\mathcal S(\rn)\times \mathcal S(\rn)\to \mathcal S'(\rn)$ be a continuous bilinear operator associated to a  complex-valued locally integrable function  $\mathcal{K}$ defined on $\re^{3n}\setminus \Delta.$ We say that $T$ is a \emph{strongly singular bilinear Calder\'on--Zygmund operator} if the following conditions hold:
\begin{enumerate}[(C1)]
\item $T$  can be extended to a bounded operator from $L^2(\rn)\times L^2(\rn)$ into $L^1(\rn);$
\item\label{item:kernelcond} there exists $0<\expo<1$ such that
\begin{equation}\label{eq:kernelcondx}
\sup_{x,x'\in\rn}\int_{\abs{x-x'}^\expo\lesssim \abs{x-y}+\abs{x-z}}\abs{\mathcal{K}(x,y,z)-\mathcal{K}(x',y,z)}\,dydz<\infty,
\end{equation}
\begin{equation}\label{eq:kernelcondy}
\sup_{y,y'\in\rn}\int_{\abs{y-y'}^\expo\lesssim \abs{y-x}+\abs{y-z}}\abs{\mathcal{K}(x,y,z)-\mathcal{K}(x,y',z)}\,dxdz<\infty,
\end{equation}
\begin{equation}\label{eq:kernelcondz}
\sup_{z,z'\in\rn}\int_{\abs{z-z'}^\expo\lesssim \abs{z-x}+\abs{z-y}}\abs{\mathcal{K}(x,y,z)-\mathcal{K}(x,y,z')}\,dxdy<\infty;
\end{equation}
\item $T,$ $T^{*1}$ and $T^{*2}$ can be extended to bounded operators from $L^{2}(\rn)\times L^2(\rn)$ into $L^{{1}/{\expo}}(\rn)$ with $\expo$ as given above.
\end{enumerate}
\end{definition}

A straightforward example of a kernel that satisfies \eqref{eq:kernelcondx} is the following: Let $\mathcal{K}$ be a locally integrable function defined on $\re^{3n}\setminus \Delta$ such that there   exist $0<s\le 1$ and $0 < \expo<1$ for which
\begin{equation}\label{eq:kernelLip}
\abs{\mathcal{K}(x,y,z)-\mathcal{K}(x',y,z)}\lesssim \frac{\abs{x-x'}^s}{(\abs{x-y}+\abs{x-z})^{2n+\frac{s}{\expo}}}\quad \text{if }  \abs{x-x'}^\expo\lesssim \abs{x-y}+\abs{x-z}.
\end{equation}
Similar statements follow in relation to   \eqref{eq:kernelcondy} and  \eqref{eq:kernelcondz}.

\begin{remark}\label{re:Tpointwise} We note that if $T:\sw\times\sw\to\swp$ is a bilinear continuous operator
 associated to a kernel $\mathcal{K}$ that can be extended to a bounded operator from $L^2(\rn)\times L^2(\rn)$ into $L^r(\rn)$ for some $1\le r<\infty$ then \eqref{eq:Tpointwise} holds for $f,g\in L^\infty_c(\rn)$ and for almost every $x\notin \text{supp}\,(f)\cap\text{supp}\, (g).$  This follows from a limiting argument and will be implicitly used throughout the proofs.
\end{remark}

We next state and prove the main results in this section.

\begin{theorem}\label{thm:boundgral} 
Let $T$ be a strongly singular bilinear Calder\'on--Zygmund operator. Then $T$ can be extended to a bounded operator from $L^\infty(\rn)\times L^\infty(\rn)$ into $BMO(\rn),$ from $\mathcal{H}^1(\rn)\times L^\infty(\rn)$ into $L^1(\rn),$  from $ L^\infty(\rn)\times \mathcal{H}^1(\rn)$ into $L^1(\rn)$ and from $L^p(\rn)\times L^q(\rn)$ into $L^r(\rn)$ for all $p,q,r$ satisfying $1\le p,q\le \infty,$ $1\le r<\infty,$ ${1}/{r}={1}/{p}+{1}/{q},$ $({1}/{p},{1}/{q},{1}/{r})\neq (1,0,1)$ and $({1}/{p},{1}/{q},{1}/{r})\neq (0,1,1).$
\end{theorem}

The proof of Theorem~\ref{thm:boundgral}, detailed at the end of this section, will be a consequence of the following theorem, duality and bilinear interpolation.

\begin{theorem}\label{thm:boundBMO}
Let $T:\sw\times\sw\to\swp$ be a bilinear continuous operator associated to a  complex-valued locally integrable function  $\mathcal{K}$ defined on $\re^{3n}\setminus \Delta$ that verifies condition \eqref{eq:kernelcondx} for some $0<\expo<1$. Assume also that $T$ can be extended to a bounded operator from
 $L^2(\rn)\times L^2(\rn)$ into $L^1(\rn)$ (or from $L^2(\rn)\times L^\infty(\rn)$ into $L^2(\rn)$ or from $L^\infty(\rn)\times L^2(\rn)$ into $L^2(\rn)$) and from $L^2(\rn)\times L^2(\rn)$ into $L^{{1}/{\expo}}(\rn).$
 Then $T$ can be extended to a bounded operator from $L^\infty(\rn)\times L^\infty(\rn)$ into $BMO(\rn).$
\end{theorem}

\begin{proof} Note that $T$ is well-defined on $L^\infty_c(\rn)\times L^\infty_c(\rn)$ since $T$ is defined on $L^2(\rn)\times L^2(\rn)$ and recall that, as pointed out in Remark~\ref{re:Tpointwise}, \eqref{eq:Tpointwise} holds for $f,g\in L^\infty_c(\rn)$ and for almost every $x\notin \text{supp}\,(f)\cap\text{supp}\, (g).$ 
We will prove  that for all cubes $Q\subset \rn$ there is $C_Q\in\com$ such that
\begin{equation*}
\dashint_Q\abs{T(f,g)(x)-C_Q}\,dx\lesssim \norm{f}{L^\infty}\norm{g}{L^\infty} \quad \forall f,g\in L^\infty_c(\rn),
\end{equation*}
where the implicit constant is independent of $Q.$
 This gives that $T$ is bounded  from $L^\infty_c(\rn)\times L^\infty_c(\rn)$ into $BMO(\rn).$ We refer the reader to Appendix~\ref{app} regarding how   $T$ can be extended to a bounded operator from $L^\infty(\rn)\times L^\infty(\rn)$ into $BMO(\rn).$

Fix  a  cube $Q$   contained in $\rn$ of side length $d>0$ and center $\qc.$ Define $\qtil$ as the cube with center $\qc$ and side length $2d^{\expo}$ if $d\le 1$ or side length $2d$ if $d>1;$ note that  $Q\subset \qtil.$ For $f, g\in L^{\infty}_c(\rn)$, we write
\[
T(f,g)=T(f\chi_{\qtil},g\chi_\qtil)+T(f\chi_\qtil,g\chi_{\qtil^c})+T(f\chi_{\qtil^c},g\chi_\qtil)+T(f\chi_{\qtil^c},g\chi_{\qtil^c}).
\]

We first estimate the term $T(f\chi_\qtil,g\chi_\qtil).$ If $d\leq 1$ apply H\"older's inequality, the boundedness of $T$ from $L^2(\rn)\times L^2(\rn)$ into $L^{{1}/{\expo}}(\rn)$ and the fact that $|\qtil|\sim\abs{Q}^\expo$ to get that
\begin{align*}
\dashint_Q\abs{T(f\chi_\qtil,g\chi_\qtil)(x)}\,dx&\le \left(\dashint_Q\abs{T(f\chi_\qtil,g\chi_\qtil)(x)}^{\frac{1}{\expo}}\,dx\right)^{\expo}\\
&\lesssim  \left(\dashint_\qtil\abs{f(y)}^2\,dy\right)^{\frac{1}{2}} \left(\dashint_\qtil\abs{g(z)}^2\,dz\right)^{\frac{1}{2}}\\
&\lesssim  \norm{f}{L^\infty}\norm{g}{L^\infty}.
\end{align*}
When $d>1$ we proceed similarly using the boundedness of $T$ from $L^2(\rn)\times L^2(\rn)$ into $L^1(\rn)$ (or from $L^2(\rn)\times L^\infty(\rn)$ into $L^2(\rn)$ or from $L^\infty(\rn)\times L^2(\rn)$ into $L^2(\rn)$)  and that $|\qtil|\sim \abs{Q}$, which implies
\begin{align*}
\dashint_Q\abs{T(f\chi_\qtil,g\chi_\qtil)(x)}\,dx
&\lesssim \norm{f}{L^\infty}\norm{g}{L^\infty}.
\end{align*}

We will next prove that
\begin{align}
\dashint_Q\abs{T(f\chi_\qtil,g\chi_{\qtil^c})(x)-C_{Q,1}}\,dx
&\lesssim  \norm{f}{L^\infty}\norm{g}{L^\infty},\label{eq:est1}\\
\dashint_Q\abs{T(f\chi_{\qtil^c},g\chi_\qtil)(x)-C_{Q,2}}\,dx
&\lesssim  \norm{f}{L^\infty}\norm{g}{L^\infty},\label{eq:est2}\\
\dashint_Q\abs{T(f\chi_{\qtil^c},g\chi_{\qtil^c})(x)-C_{Q,3}}\,dx
&\lesssim  \norm{f}{L^\infty}\norm{g}{L^\infty},\label{eq:est3}
\end{align}
with the implicit constants independent of $f,g$ and $Q,$ and where
\begin{align*}
C_{Q,1}&:=\int_{\rtn} {\mathcal{K}}(\qc,y,z) f(y)\chi_\qtil(y) g(z)\chi_{\qtil^c}(z)\,dydz,\\
C_{Q,2}&:=\int_{\rtn} {\mathcal{K}}(\qc,y,z) f(y)\chi_{\qtil^c}(y) g(z)\chi_\qtil(z)\,dydz,\\
C_{Q,3}&:=\int_{\rtn} {\mathcal{K}}(\qc,y,z) f(y)\chi_{\qtil^c}(y) g(z)\chi_{\qtil^c}(z)\,dydz.
\end{align*}
We will show \eqref{eq:est1}, with the estimates \eqref{eq:est2} and \eqref{eq:est3} following in the same way.
Note that if $x\in Q,$ $y\in \qtil$ and $z\in \qtil^c$ then
\begin{align*}
|x-\qc|^\expo&\lesssim d^\expo \lesssim  |\qc-z|\lesssim |\qc-z|+|\qc-y|\quad \text{if } d\le1,\\
|x-\qc|^\expo&\lesssim d^\expo\le d\lesssim |\qc-z|\lesssim |\qc-z|+|\qc-y|\quad \text{if } d>1.
\end{align*}
Therefore, for $x\in Q$ and using \eqref{eq:kernelcondx}, we have
\begin{align*}
|T(f\chi_\qtil, g\chi_{\qtil^c})&(x)-C_{Q,1}|\\
&\lesssim \int_{y\in \qtil,z\in \qtil^c}\abs{\mathcal{K}(x,y,z)-\mathcal{K}(\qc,y,z)}\abs{f(y)}\abs{g(z)}\,dydz\\
&\lesssim  \int_{|x-\qc|^\expo\lesssim |\qc-z|+|\qc-y|}\abs{\mathcal{K}(x,y,z)-\mathcal{K}(\qc,y,z)}\abs{f(y)}\abs{g(z)}\,dydz\\
&\lesssim \norm{f}{L^\infty}\norm{g}{L^\infty}.
\end{align*}
Averaging in $x$ over $Q$ we obtain \eqref{eq:est1}. 

\end{proof}

We end this section with the proof of Theorem~\ref{thm:boundgral}.

\begin{proof}[Proof of Theorem~\ref{thm:boundgral}]
Since $T$ can be extended to a  bounded operator from  $L^2(\rn)\times L^2(\rn)$ into $L^1(\rn),$ a duality argument implies that $T^{*1}$ and $T^{*2}$ can be extended to bounded operators from $L^\infty(\rn)\times L^2(\rn)$ into $L^2(\rn)$ and from  $L^2(\rn)\times L^\infty(\rn)$ into $L^2(\rn),$ respectively.
Theorem~\ref{thm:boundBMO} then gives that  $T,$ $T^{*1}$ and $T^{*2}$ can be extended to  bounded operators from $L^\infty(\rn)\times L^\infty(\rn)$ into $BMO(\rn).$  As a consequence, by duality, $T$ can be extended to a bounded operator from $\mathcal{H}^1(\rn)\times L^\infty(\rn)$ into $L^1(\rn)$ and from
$ L^\infty(\rn)\times \mathcal{H}^1(\rn)$ into $L^1(\rn).$ Bilinear complex interpolation implies that $T$ is bounded from $L^p(\rn)\times L^q(\rn)$ into $L^r(\rn)$ for all $p,q,r$ such that $({1}/{p},{1}/{q},{1}/{r})$ is in the convex hull of the points $(0,0,0),$ $(1,0,1)$ and $(0,1,1).$ That is, $1\le p,q\le \infty,$ $1\le r<\infty,$ ${1}/{r}={1}/{p}+{1}/{q},$ $({1}/{p},{1}/{q},{1}/{r})\neq (1,0,1)$ and $({1}/{p},{1}/{q},{1}/{r})\neq (0,1,1).$
\end{proof}

\section{Pointwise inequalities and weighted estimates for strongly singular bilinear Calder\'on--Zygmund operators}\label{sec:maxweight}

In this section we  show that Theorem~\ref{thm:boundBMO} can be improved if the stronger condition \eqref{eq:kernelLip} is assumed, by proving a pointwise inequality in terms of the sharp maximal operator and the Hardy--Littlewood maximal operator. In particular, such result will imply weighted estimates for strongly singular bilinear Calder\'on--Zygmund operators that satisfy  \eqref{eq:kernelLip}.

\begin{theorem}\label{thm:improvement}
Let $T:\sw\times\sw\to\swp$ be a bilinear continuous operator associated to a  complex-valued locally integrable function  $\mathcal{K}$ defined on $\re^{3n}\setminus \Delta$ that verifies condition \eqref{eq:kernelLip} for some $0<\expo<1$ and $0<s\le 1.$ Assume also that $T$ can be extended to a bounded operator from
 $L^2(\rn)\times L^2(\rn)$ into $L^1(\rn)$ and from $L^2(\rn)\times L^2(\rn)$ into $L^{{1}/{\expo}}(\rn).$ Then $T$ satisfies
\begin{equation}\label{eq:maxineq}
\ms(T(f,g))(x)\lesssim \hld(f)(x)\hld(g)(x) \quad \forall x\in\rn,\, f,g\in L^\infty_c(\rn),
\end{equation}
where $\hld(h)=\left(\mathcal{M}(\abs{h}^2)\right)^{{1}/{2}}.$
\end{theorem}

\begin{corollary}\label{coro:weights} If $T:\sw\times\sw\to\swp$ satisfies the hypothesis of Theorem~\ref{thm:improvement}, the following statements hold true:
\begin{enumerate}[(a)]
\item \label{item:a} $T$ can be extended to a  bounded operator from $L^\infty(\rn)\times L^\infty(\rn)$ into $BMO(\rn).$
\item \label{item:b}Let $2<p,q<\infty$ and $r$ be such that $1/r=1/p+1/q;$ consider $v\in A_{p/2},$ $w\in A_{q/2}$ and define $u:=v^{r/p}w^{r/q}.$ Then $T$ can be extended to a bounded operator from $L^{p}_v(\rn)\times L^q_w(\rn)$ into $L^r_u(\rn).$
\end{enumerate}
\end{corollary}

\begin{remark}
If $T$ is a strongly singular bilinear Calder\'on--Zygmund operator that satisfies \eqref{eq:kernelLip}, then, besides the results from Theorem~\ref{thm:boundgral}, $T$ satisfies \eqref{eq:maxineq} and the weighted estimates stated in item  \eqref{item:b} of Corollary~\ref{coro:weights}.  If $T$ satisfies \eqref{eq:kernelLip} and its symmetric counterparts, then \eqref{eq:maxineq} and  item  \eqref{item:b} of Corollary~\ref{coro:weights} also hold for  $T^{*1}$ and $T^{*2}.$
\end{remark}

We briefly comment on the proof of Corollary~\ref{coro:weights} and then prove Theorem~\ref{thm:improvement}.
Note that   item \eqref{item:a} of Corollary~\ref{coro:weights} follows from Theorem~\ref{thm:boundBMO}, but of course, it can be inferred from \eqref{eq:maxineq}: such inequality gives that $T_\sigma$ is bounded from $L^\infty_c(\rn)\times L^\infty_c(\rn)$ into $BMO(\rn),$ which in turn implies  that $T_\sigma$ can be extended to a bounded operator from $L^\infty(\rn)\times L^\infty(\rn)$ into $BMO(\rn)$ as shown in Appendix~\ref{app}. For   \eqref{item:b} we use the  weighted Fefferman--Stein inequality (see Fefferman--Stein~\cite{MR0447953} and Cruz-Uribe--Martell--P\'erez~\cite[Theorem 1.3]{MR2078632}) to get that
\begin{equation*}
\norm{T_\sigma(f,g)}{L^r_u}\le \norm{\hl(T_\sigma(f,g))}{L^r_u}\lesssim \norm{\ms(T_\sigma(f,g))}{L^r_u};
\end{equation*}
we then use \eqref{eq:maxineq}, H\"older's inequality and the boundedness properties in weighted Lebesgue spaces of $\hld$ to obtain
$$
\norm{T_\sigma(f,g)}{L^r_u}\lesssim \norm{f}{L^p_v}\norm{g}{L^q_w}\quad \forall f,g\in L^\infty_c(\rn).
$$
$T_\sigma$ can then be extended by density to a bounded operator from $L^p_v(\rn)\times L^q_w(\rn)$ into $L^r_u(\rn).$

\begin{proof}[Proof of Theorem~\ref{thm:improvement}] We follow  ideas from  \cite[Theorem 2.2]{MR3411149}.
We have to prove  that for all cubes $Q\subset \rn$ there is a constant $C_Q\in\com$ such that
\begin{equation*}
\dashint_Q\abs{T(f,g)(y)-C_Q}\,dy\lesssim \hld(f)(x)\hld(g)(x) \quad\forall x\in Q, f,g\in L^\infty_c(\rn),
\end{equation*}
where the implicit constant is independent of $Q.$

Fix  a  cube $Q\subset \rn$ of side length $d>0$ and center $\qc;$ let $\qtil$ be as in the proof of Theorem \ref{thm:boundBMO}. For $f, g\in L^\infty_c(\rn)$ consider
\[
T(f,g)=T(f\chi_{\qtil},g\chi_\qtil)+T(f\chi_\qtil,g\chi_{\qtil^c})+T(f\chi_{\qtil^c},g\chi_\qtil)+T(f\chi_{\qtil^c},g\chi_{\qtil^c}).
\]

The term $T(f\chi_\qtil,g\chi_\qtil)$ is controlled using the boundedness of $T$ in the same way done in the proof of  Theorem~\ref{thm:boundBMO}.
If $d\leq 1$, we apply H\"older's inequality and the boundedness of $T$ from $L^2(\rn)\times L^2(\rn)$ into $L^{{1}/{\expo}}(\rn)$ to get that for all $x\in  Q$
\begin{align*}
\dashint_Q\abs{T(f\chi_\qtil,g\chi_\qtil)(v)}\,dv&\lesssim  \left(\dashint_\qtil\abs{f(y)}^2\,dy\right)^{\frac{1}{2}} \left(\dashint_\qtil\abs{g(z)}^2\,dz\right)^{\frac{1}{2}}&\lesssim \hld(f)(x)\,\hld(g)(x).
\end{align*}
When $d>1$ we proceed similarly using the boundedness of $T$ from $L^2(\rn)\times L^2(\rn)$ into $L^1(\rn)$, which implies
\begin{align*}
\dashint_Q\abs{T(f\chi_\qtil,g\chi_\qtil)(v)}\,dv
&\lesssim \hld(f)(x)\,\hld(g)(x)\quad \forall x\in Q.
\end{align*}

For the terms $T(f\chi_\qtil,g\chi_{\qtil^c}),$ $T(f\chi_{\qtil^c},g\chi_\qtil)$ and $T(f\chi_{\qtil^c},g\chi_{\qtil^c}),$ we will  prove that
\begin{align*}
\dashint_Q\abs{T(f\chi_\qtil,g\chi_{\qtil^c})(v)-C_{Q,1}}\,dv
&\lesssim  \hld(f)(x)\,\hld(g)(x)\quad \forall x\in Q,\\
\dashint_Q\abs{T(f\chi_{\qtil^c}g\chi_\qtil)(v)-C_{Q,2}}\,dv
&\lesssim  \hld(f)(x)\,\hld(g)(x)\quad \forall x\in Q,\\
\dashint_Q\abs{T(f\chi_{\qtil^c}g\chi_{\qtil^c})(v)-C_{Q,3}}\,dv
&\lesssim  \hld(f)(x)\,\hld(g)(x)\quad \forall x\in Q,
\end{align*}
where $C_{Q,1}, C_{Q,2}, C_{Q,3}$ are as in the proof of Theorem~\ref{thm:boundBMO}.

Define $\mathcal{L}_Q(v,y,z)={\mathcal{K}}(v,y,z)-{\mathcal{K}}(\qc,y,z);$ set  $t=\expo$ if $d<1$ and $t=1$ if $d>1.$ For $v,x\in Q$ we have that
\begin{align*}
&\abs{T(f\chi_\qtil, g\chi_{\qtil^c})(v)-C_{Q,1}}\lesssim\, \sum_{k=1}^\infty\mathop{\int_{y\in\qtil}}_{\abs{z-\qc}\sim (2^k d)^{t}}\abs{\mathcal{L}_Q(v,y,z)}\abs{f(y)}\abs{g(z)}\,dydz\\
&\lesssim  \sum_{k=1}^\infty\left(\mathop{\int_{y\in\qtil}}_{\abs{z-\qc}\sim (2^k d)^{t}}\abs{\mathcal{L}_Q(v,y,z)}^2\,dydz\right)^{\frac{1}{2}} \left(\int_\qtil\abs{f(y)}^2\,dy \right)^{\frac{1}{2}}  \left(\mathop{\int}_{\abs{z-\qc}\sim (2^k d)^{t}}\abs{g(z)}^2\,dz \right)^{\frac{1}{2}}\\
&\lesssim  \,\hld(f)(x)\,\hld(g)(x) \sum_{k=1}^\infty d^{n t} 2^{k\frac{n}{2}t} \left(\mathop{\int_{y\in\qtil}}_{\abs{z-\qc}\sim (2^k d)^{t}}\abs{\mathcal{L}_Q(v,y,z)}^2\,dydz\right)^{\frac{1}{2}}.
\end{align*}
We next show that the sum above is bounded by a constant independent of $d,$ $v$ and $Q.$ Indeed, we have
\begin{align*}
\mathop{\int_{y\in\qtil}}_{\abs{z-\qc}\sim (2^k d)^{t}}\abs{\mathcal{L}_Q(v,y,z)}^2\,dydz&\lesssim \mathop{\int_{y\in\qtil}}_{\abs{z-\qc}\sim (2^k d)^{t}}\left(\frac{\abs{v-\qc}^s}{(\abs{\qc-y}+\abs{\qc-z})^{2n+\frac{s}{\expo}}}\right)^2\,dydz\\
&\sim \abs{v-\qc}^{2s} |\qtil|\, (2^k d)^{-2t(2n+\frac{s}{\expo})+tn}\\
&\lesssim d^{2s} d^{tn} (2^k d)^{-2t(2n+\frac{s}{\expo})+tn},
\end{align*}
where in the first inequality we have used that $\abs{z-\qc}\sim (2^k d)^{t}$ implies that $\abs{v-\qc}^\expo\lesssim \abs{\qc-y}+\abs{\qc-z}$ since $v\in Q$ and in the third inequality we have used that $v\in Q$  and the definition of $\qtil.$ Then, the sum above is controlled by
\[
 \sum_{k=1}^\infty d^{n t} 2^{k\frac{n}{2}t} ( d^{2s} d^{tn} (2^k d)^{-2t(2n+\frac{s}{\expo})+tn})^{\frac{1}{2}}= \sum_{k=1}^\infty  d^{s(1-\frac{t}{\expo})} 2^{-kt(n+\frac{s}{\expo})}\le \sum_{k=1}^\infty   2^{-kt(n+\frac{s}{\expo})} <\infty.
\]
We have then obtained that
\begin{align*}
\abs{T(f\chi_\qtil, g\chi_{\qtil^c})(v)-C_{Q,1}}\lesssim \hld(f)(x)\,\hld(g)(x)\quad \forall v,x\in Q.
\end{align*}
Integrating in $v$ over $Q$ it follows that
\begin{align*}
\dashint_Q\abs{T(f\chi_\qtil, g\chi_{\qtil^c})(v)-C_{Q,1}}\,dv\lesssim \hld(f)(x)\,\hld(g)(x)\quad \forall x\in Q,
\end{align*}
as desired. The term $T(f\chi_{\qtil^c}, g\chi_\qtil)(v)$ is treated analogously.  For the term $T(f\chi_{\qtil^c}, g\chi_{\qtil^c})$ we proceed  similarly as above and get that if  $v,x\in Q$ then
\begin{align*}
&\abs{T(f\chi_{\qtil^c}, g\chi_{\qtil^c})(v)-C_{Q,3}} \lesssim \,\sum_{k=1}^\infty\sum_{j=1}^\infty\mathop{\int_{\abs{y-\qc}\sim (2^j d)^{t}}}_{\abs{z-\qc}\sim (2^k d)^{t}}\abs{\mathcal{L}_Q(v,y,z)}\abs{f(y)}\abs{g(z)}\,dydz\\
&\lesssim \hld(f)(x)\,\hld(g)(x) \sum_{k=1}^\infty \,\sum_{j=1}^\infty d^{nt} 2^{(j+k)\frac{n}{2}t} \left(\mathop{\int_{\abs{y-\qc}\sim (2^j d)^{t}}}_{\abs{z-\qc}\sim (2^k d)^{t}}\abs{\mathcal{L}_Q(v,y,z)}^2\,dydz\right)^{\frac{1}{2}}.
\end{align*}
This last sum is finite and controlled by a constant independent of $d,$ $v$ and $Q$ as the following computation shows. We have
\begin{align*}
&\mathop{\int_{\abs{y-\qc}\sim (2^j d)^{t}}}_{\abs{z-\qc}\sim (2^k d)^{t}}\abs{\mathcal{L}_Q(v,y,z)}^2\,dydz\lesssim \mathop{\int_{\abs{y-\qc}\sim (2^j d)^{t}}}_{\abs{z-\qc}\sim (2^k d)^{t}}\left(\frac{\abs{v-\qc}^s}{(\abs{\qc-y}+\abs{\qc-z})^{2n+\frac{s}{\expo}}}\right)^2\,dydz\\
&\quad\quad\sim \abs{v-\qc}^{2s} \left(\int_{\abs{y-\qc}\sim (2^j d)^{t}}\frac{1}{\abs{\qc-y}^{2n+\frac{s}{\expo}}}\,dy\right)
\left(\int_{\abs{z-\qc}\sim (2^k d)^{t}}\frac{1}{\abs{\qc-z}^{2n+\frac{s}{\expo}}}\,dz\right)\\
&\quad\quad\lesssim d^{2s} (2^j d)^{-t(n+\frac{s}{\expo})} (2^k d)^{-t(n+\frac{s}{\expo})},
\end{align*}
which gives that the sum above is controlled by
\begin{align*}
\sum_{k=1}^\infty \,\sum_{j=1}^\infty d^{nt} 2^{(j+k)\frac{n}{2}t} (d^{2s} (2^j d)^{-t(n+\frac{s}{\expo})} (2^k d)^{-t(n+\frac{s}{\expo})})^{\frac{1}{2}}&=\sum_{k=1}^\infty \,\sum_{j=1}^\infty d^{s(1-\frac{t}{\expo})} 2^{-j\frac{ts}{2\expo}} 2^{-k\frac{ts}{2\expo}}\\
&\le \sum_{k=1}^\infty \,\sum_{j=1}^\infty
 2^{-j\frac{ts}{2\expo}} 2^{-k\frac{ts}{2\expo}}<\infty.
\end{align*}
We then obtain that
\begin{align*}
\abs{T(f\chi_{\qtil^c}, g\chi_{\qtil^c})(v)-C_{Q,3}}\lesssim \hld(f)(x)\,\hld(g)(x)\quad \forall v,x\in Q,
\end{align*}
and integrating in $v$ over $Q$ gives the expected inequality.

\end{proof}

\section{The critical bilinear H\"ormander class $BS_{\rho, \delta}^{-n(1-\rho)}$}\label{sec:critical}

In this section, we show that if  $\sigma\in BS^{-n(1-\rho)}_{\rho,\delta}$ with $0\le \delta\le \rho$ and $0<\rho<1$ then the kernel of $T_\sigma$ satisfies condition \eqref{eq:kernelcondx} with $\varepsilon=\rho$. By  the symbolic calculus of the bilinear H\"ormander classes proved in \cite[Theorem 2.1]{MR2660466},  \eqref{eq:kernelcondy} and  \eqref{eq:kernelcondz} also follow with $\varepsilon=\rho$. As a consequence, $T_\sigma$ has a strongly singular bilinear Calder\'on--Zygmund kernel.

The operator $T_\sigma$ is known to be bounded from $L^2(\rn)\times L^2(\rn)$ into $L^{1}(\rn)$ (see \cite{MR3205530} and references therein); it was shown in  \cite[Lemma 2.1]{MR3411149} that if the condition $0<\rho<1/2$ is further imposed, then $T_\sigma$ and its transposes are also bounded from $L^2(\rn)\times L^2(\rn)$ into $L^{{1}/{\rho}}(\rn).$  Therefore, $T_\sigma$ with such a symbol becomes a strongly singular bilinear Calder\'on--Zygmund operator. In particular, Theorem~\ref{thm:boundgral} applied to $T_\sigma$ recovers \cite[Theorem 1.1]{MR3411149}. Note that $T_\sigma$ does not necessarily satisfies the hypothesis of Theorem~\ref{thm:improvement} since its kernel  may fail to satisfy the stronger condition \eqref{eq:kernelLip}; however, the pointwise inequality \eqref{eq:maxineq} does hold for $T_\sigma$ as proved in \cite[Theorem 2.2]{MR3411149}. Therefore, not only does $T_\sigma$ satisfy all boundedness properties stated in Theorem~\ref{thm:boundgral} but it also verifies the weighted estimated from part \eqref{item:b} of Corollary~\ref{coro:weights}.

\medskip

Fix  $\sigma\in BS^{-n(1-\rho)}_{\rho,\delta}$ with  $0\le \delta\le \rho$ and $0< \rho<1.$ We next proceed to show that the kernel of $T_\sigma,$ which we denote by $K,$ satisfies \eqref{eq:kernelcondx} with $\varepsilon=\rho.$

We begin by showing that
\begin{equation}\label{eq:kernelcondxbig}
\sup_{|x-x'|\geq 1}\int_{R_\rho}\abs{K (x,y,z)-K (x',y,z)}\,dydz<\infty,
\end{equation}
where for fixed $x,x',$ we denote
$$R_\rho:=\{(y, z)\in \mathbb R^{2n}: 3\abs{x-x'}^{\rho}\leq \abs{x-y}+\abs{x-z}\};$$
let us also write
$$R:=\{(y, z)\in \mathbb R^{2n}: \abs{x-x'}\leq \abs{x-y}+\abs{x-z}\}.$$
Fix $x, x'\in\rn $ so that $|x-x'|\geq 1$. Using \cite[Theorem E (v)]{MR3205530}, we obtain that for $(y, z)\in R_\rho$ we have
\[
|K(x, y, z)|\lesssim (\abs{x-y}+\abs{x-z})^{-n(1+1/\rho)}.
\]
Furthermore, if $(y, z)\in R$, by \cite[Theorem E (vi)]{MR3205530} we also get
\[
|K(x, y, z)-K(x', y, z)|\lesssim |x-x'|^{\rho}(\abs{x-y}+\abs{x-z})^{-n(1+1/\rho)-1}.
\]

Note that $n(1+1/\rho)=2n+n(1-\rho)/\rho>2n$. We break down the left hand side of  \eqref{eq:kernelcondxbig} into two cases. In the first case, let us further assume that $|x-x'|\leq 3^{1/(1-\rho)}$. Then, $R_\rho\subset R$ and we can write
\begin{align*}
\int_{R_\rho}\abs{K (x,y,z)-K (x',y,z)}\,dydz&\lesssim \int_{R_\rho}|x-x'|^\rho (\abs{x-y}+\abs{x-z})^{-n(1+1/\rho)-1}\,dydz\\
&\lesssim \int_{R_\rho}(\abs{x-y}+\abs{x-z})^{-n(1+1/\rho)}\,dydz\\
&\lesssim |x-x'|^{-n(1-\rho)}\lesssim 1.
\end{align*}
Secondly, let us now assume that $|x-x'|>3^{1/(1-\rho)}$; in this case $R\subset R_\rho$. Write
\[
\int_{R_\rho}\abs{K (x,y,z)-K (x',y,z)}\,dydz=I+II,
\]
with $I=\int_R \abs{K (x,y,z)-K (x',y,z)}\,dydz$ and $II=\int_{R_\rho\setminus R} \abs{K (x,y,z)-K (x',y,z)}\,dydz$.
We estimate $I$ as follows:
\begin{align*}
I&\lesssim \int_R |x-x'|^{\rho}(\abs{x-y}+\abs{x-z})^{-n(1+1/\rho)-1}\,dydz\\
&\lesssim \int_R (\abs{x-y}+\abs{x-z})^{-n(1+1/\rho)-1+\rho}\,dydz\\
&\lesssim |x-x'|^{n(1-1/\rho)-1+\rho}\lesssim 1.
\end{align*}
We estimate  $II$ as follows:
\[
II\leq \int_{R_\rho\setminus R}|K(x, y, z)|\,dydz+\int_{R_\rho\setminus R}|K(x', y, z)|\,dydz.
\]
Now,
\begin{align*}
\int_{R_\rho\setminus R}|K(x, y, z)|\,dydz &\lesssim \int_{R_\rho}(|x-y|+|x-z|)^{-n(1+1/\rho)}\,dydz\lesssim |x-x'|^{-n(1-\rho)}\lesssim 1.
\end{align*}
Note also that, if $(y ,z)\in R_\rho\setminus R$, we have $|x-y|+|x-z|<|x-x'|$, hence
\[
|x'-y|+|x'-z|\geq 2|x-x'|-(|x-y|+|x-z|)>|x-x'|.
\]
Thus, by appealing again to \cite[Theorem E (v)]{MR3205530}, we can write
\begin{align*}
\int_{R_\rho\setminus R}|K(x', y, z)|\,dydz &\lesssim \int_{|x-x'|<|x'-y|+|x'-z|} (|x'-y|+|x'-z|)^{-n(1+1/\rho)}\,dydz\\
& \lesssim |x-x'|^{n(1-1/\rho)}\lesssim 1.
\end{align*}
This finishes the proof of \eqref{eq:kernelcondxbig}.

It remains to show that the following estimate also holds:
\begin{equation}\label{eq:kernelcondxsmall}
\sup_{0<|x-x'|<1}\int_{R_\rho}\abs{K (x,y,z)-K (x',y,z)}\,dydz<\infty.
\end{equation}
Note that the approach used for the case $|x-x'|>1$ using the pointwise estimates for $K$ does not work anymore since this  leads to negative powers of $|x-x'|.$ In order to prove  \eqref{eq:kernelcondxsmall} we follow the idea in \cite{MR849442}; see also \cite[Chapter VII, p. 322]{MR1232192}. Let $\varphi(\xi,\eta)$  and $\psi(\xi, \eta),$ $\xi,\eta\in\rn,$ be infinitely differentiable functions so that $\text{supp}\, \varphi\subset \{(\xi,\eta):  |\xi|+|\eta|\leq 2\},$ $\text{supp}\, \psi\subset \{(\xi,\eta): 1/2\leq |\xi|+|\eta|\leq 2\}$ and $\varphi(\xi,\eta)+\sum_{j=1}^\infty \psi (2^{-j}\xi, 2^{-j}\eta)=1$ for all $\xi,\eta\in\rn.$ Fix $x, x'\in\rn$ so that $|x-x'|<1$. For $j\ge 0,$ we write
\[
K_j (x, y, z)=\int_{\rtn} \sigma_j (x, \xi, \eta)e^{i(x-y)\cdot\xi}e^{i(x-z)\cdot\eta}\,d\xi d\eta,
\]
where
\[
\sigma_0 (x, \xi, \eta)=\sigma (x, \xi, \eta)\varphi (\xi, \eta)\quad \text{and}\quad \sigma_j (x, \xi, \eta)=\sigma (x, \xi, \eta)\psi (2^{-j}\xi, 2^{-j}\eta) \,\, \text{for }j\ge 1.
\]
It suffices to estimate
\[
\int_{R_\rho}\abs{K_j (x,y,z)-K_j (x',y,z)}\,dydz\lesssim c_j (x-x'),
\]
where the functions $c_j$ satisfy $\sum_{j=0}^\infty c_j(x-x')\lesssim 1$ uniformly in $x, x'$.

As a first instance, we estimate crudely
\[
\int_{R_\rho}\abs{K_j (x,y,z)-K_j (x',y,z)}\,dydz\leq \int_{R_\rho}\abs{K_j (x,y,z)}\,dydz+\int_{R_\rho}\abs{K_j (x',y,z)}\,dydz.
\]
Note that, since $|x-x'|^\rho\geq |x-x'|$, if $(y, z)\in R_\rho$, then
$$|x'-y|+|x'-z|\geq |x-y|+|x-z|-2|x-x'|\geq 3|x-x'|^\rho-2|x-x'|\geq |x-x'|^\rho.$$
Thus, it suffices to estimate $\int_{R_\rho}\abs{K_j (x,y,z)}\,dydz$.

Let $N\in 2\mathbb N$ be so that $N>n$. By using H\"older's inequality,  the integral of $\abs{K_j (x,y,z)}$ over $R_\rho$ is less than
\[
\|(|x-y|^2+|x-z|^2)^{N/2}K_j(x, y, z)\|_{L^2(dydz)}\|(|x-y|^2+|x-z|^2)^{-N/2}\chi_{_{R_\rho}}\|_{L^2(dydz)}.
\]
Using polar coordinates immediately gives that
\[
\|(|x-y|^2+|x-z|^2)^{-N/2}\chi_{_{R_\rho}}\|_{L^2(dydz)}\sim |x-x'|^{\rho(n-N)}.
\]
We now estimate the first $L^2$ norm above. Integration by parts gives
\begin{align*}
&(|x-y|^2+|x-z|^2)^{N/2}K_j(x, y, z)=(-1)^{N/2}\int_{\rtn} \Delta^{N/2}_{\xi, \eta}\sigma_j(x, \xi, \eta)e^{i(x-y)\cdot\xi+i(x-z)\cdot\eta}\,d\xi d\eta\\
&=\sum_{|\alpha_1+\alpha_2|=N} c_{\alpha_1, \alpha_2}\int_{\rtn} \partial_{\xi, \eta}^{\alpha_1}\sigma (x, \xi, \eta)2^{-j|\alpha_2|}(\partial_{\xi, \eta}^{\alpha_2}\psi)(2^{-j}\xi, 2^{-j}\eta)e^{i(x-y)\cdot\xi+i(x-z)\cdot\eta}\,d\xi d\eta,
\end{align*}
where $\psi$ should be replaced by $\varphi$ if $j=0.$ Now, using Plancherel's theorem, the conditions on the symbol $\sigma$ and the supports of $\varphi$ and $\psi$, we can estimate
\begin{align*}
\|(|x-y|^2+|x-z|^2)^{N/2}K_j(x, y, z)\|_{L^2(dydz)}&\lesssim \sum_{|\alpha_1+\alpha_2|=N}(2^j)^{n(\rho-1)-\rho|\alpha_1|}2^{-j\rho|\alpha_2|}2^{jn}\\&\sim 2^{j\rho(n-N)}.
\end{align*}
We have thus obtained  that
\begin{equation}
\label{firstestimate}
\int_{R_\rho}\abs{K_j (x,y,z)-K_j (x',y,z)}\,dydz\lesssim (2^j|x-x'|)^{\rho(n-N)}\quad \forall j\in\na_0.
\end{equation}

It turns out that we can improve the estimate \eqref{firstestimate} if we further require $2^{j}|x-x'|\leq 1$; we show this next. As in the above calculation, we let $N\in 2\mathbb N$ be so that $N>n$ and apply H\"older's inequality to get
$$\int_{R_\rho}\abs{K_j (x,y,z)-K_j (x',y,z)}\,dydz\leq I_1I_2,$$ where
\[
I_1:=\|(1+2^{j\rho N}(|x-y|^2+|x-z|^2)^{N/2})[K_j(x, y, z)-K_j (x', y, z)]\|_{L^2(dydz)}
\]
and
\[
I_2:=\|(1+2^{j\rho N}(|x-y|^2+|x-z|^2)^{N/2})^{-1}\chi_{_{R_\rho}}\|_{L^2(dydz)}.
\]
Using a change of variables we see that
\[
I_2\lesssim \Big(\int_{\rtn}(1+|y|^2+|z|^2)^{-N}2^{-2jn\rho}\,dydz\Big)^{\frac{1}{2}}\sim 2^{-jn\rho}.
\]
We estimate $I_1$ next. Observe first that
\[
K_j(x, y, z)-K_j (x', y, z)=I+II,
\]
where
\[
I=\int_{\rtn} \sigma_j(x', \xi, \eta)e^{i(x-y)\cdot \xi+i(x-z)\cdot \eta}(1-e^{i(x'-x)\cdot (\xi+\eta)})\,d\xi d\eta
\]
and
\[
II=\int_{\rtn}(\sigma_j (x, \xi, \eta)-\sigma_j (x', \xi, \eta))e^{i(x-y)\cdot\xi+i(x-z)\cdot\eta}\,d\xi d\eta.
\]
Therefore, we have $I_1\leq I_{1, 1}+I_{1, 2}$, with
\[
I_{1, 1}=\|(1+2^{j\rho N}(|x-y|^2+|x-z|^2)^{N/2})I\|_{L^2(dydz)}
\]
and
\[
I_{1, 2}=\|(1+2^{j\rho N}(|x-y|^2+|x-z|^2)^{N/2})II\|_{L^2(dydz)}.
\]
Integration by parts gives that $(1+2^{j\rho N}(|x-y|^2+|x-z|^2)^{N/2})I$ equals
\[
\int_{\rtn} (1+ (-1)^{N/2} 2^{j\rho N} \Delta_{\xi, \eta}^{N/2})[\sigma_j(x', \xi, \eta)(1-e^{i(x'-x)\cdot (\xi+\eta)})]e^{i(x-y)\cdot \xi+i(x-z)\cdot \eta}\,d\xi d\eta.
\]
Now, since $2^j|x-x'|\leq 1$ and $|\xi|+|\eta|\sim 2^j$ ($\lesssim 1$ if $j=0$) in the support of $\sigma_j$, we have
$$|1-e^{i(x'-x)\cdot(\xi+\eta)}|\leq |x-x'||\xi+\eta|\lesssim 2^j|x-x'|\leq (2^j|x-x'|)^\rho.$$
The above and the fact that $\sigma\in BS^{-n(1-\rho)}_{\rho, \delta}$ leads to
$$|\sigma_j (x', \xi, \eta)(1-e^{i(x'-x)\cdot(\xi+\eta)})|\lesssim 2^{-jn(1-\rho)} (2^j |x-x'|)^\rho.$$
Moreover,  $2^{j\rho N}\Delta_{\xi, \eta}^{N/2}[\sigma_j(x', \xi, \eta)(1-e^{i(x'-x)\cdot (\xi+\eta)})]$ is given by
\begin{align*}
&2^{j\rho N}\mathop{\sum_{|\alpha_1+\alpha_2+\alpha_3|=N}}_{|\alpha_3|>0} c_{\alpha_1, \alpha_2, \alpha_3}\partial_{\xi, \eta}^{\alpha_1}\sigma (x', \xi, \eta)2^{-j|\alpha_2|}(\partial_{\xi, \eta}^{\alpha_2}\psi)(2^{-j}\xi, 2^{-j}\eta)(x-x')^{\alpha_3}e^{i(x'-x)\cdot (\xi+\eta)}\\
&+2^{j\rho N}\sum_{|\alpha_1+\alpha_2|=N} c_{\alpha_1, \alpha_2, \alpha_3}\partial_{\xi, \eta}^{\alpha_1}\sigma (x', \xi, \eta)2^{-j|\alpha_2|}(\partial_{\xi, \eta}^{\alpha_2}\psi)(2^{-j}\xi, 2^{-j}\eta)(1-e^{i(x'-x)\cdot (\xi+\eta)}),
\end{align*}
where $\psi$ should be replaced by $\varphi$ if $j=0.$
 Since $0<\rho< 1$ and $2^j\abs{x-x'}\le 1$ we can control
\begin{align*}
|2^{j\rho N}\Delta_{\xi, \eta}^{N/2}&[\sigma_j(x', \xi, \eta)(1-e^{i(x'-x)\cdot (\xi+\eta)})]|\\
\lesssim& \,2^{j\rho N}\mathop{\sum_{|\alpha_1+\alpha_2+\alpha_3|=N}}_{|\alpha_3|>0} (2^j)^{-n(1-\rho)-\rho |\alpha_1|}2^{-j\rho |\alpha_2|}(2^j|x-x'|)^{|\alpha_3|}2^{-j\rho |\alpha_3|}\\
&+2^{j\rho N}\sum_{|\alpha_1+\alpha_2|=N}  (2^j)^{-n(1-\rho)-\rho |\alpha_1|}2^{-j\rho |\alpha_2|}(2^j|x-x'|)^\rho\\
\lesssim& \,(2^j|x-x'|)^\rho 2^{-jn(1-\rho)}.
\end{align*}
The previous estimates imply that
\[
|(1+ (-1)^{N/2}2^{j\rho N}\Delta_{\xi, \eta}^{N/2})[\sigma_j(x', \xi, \eta)(1-e^{(x'-x)(\xi+\eta)})]|\lesssim (2^j|x-x'|)^\rho 2^{-jn(1-\rho)};
\]
thus, by using Plancherel's theorem, we conclude that
\[
I_{1, 1}\lesssim (2^j|x-x'|)^\rho 2^{-jn(1-\rho)}2^{jn}\sim (2^j|x-x'|)^\rho 2^{jn\rho}.
\]
A similar argument shows that $I_{1, 2}\lesssim (2^j|x-x'|)^\rho 2^{jn\rho}$; the difference is that now we control
\begin{align*}
|(1+(-1)^{N/2}2^{j\rho N}\Delta_{\xi, \eta}^{N/2})(\sigma_j (x, \xi, \eta)-\sigma_j(x', \xi, \eta))|&\lesssim |x-x'|(2^j)^{-n(1-\rho)+\delta}\\
&\leq (2^j|x-x'|)^\rho 2^{-jn(1-\rho)}.
\end{align*}
All in all, we have shown that
$
I_1\lesssim (2^j|x-x'|)^\rho 2^{jn\rho},
$
which gives, for $2^j|x-x'|<1$, the improved estimate
\begin{equation}
\label{secondestimate}
\int_{R_\rho}\abs{K_j (x,y,z)-K_j (x',y,z)}\,dydz\lesssim (2^j|x-x'|)^\rho 2^{jn\rho}2^{-jn\rho}\sim (2^j|x-x'|)^{\rho}.
\end{equation}

Letting $c_j (x-x'):=\min\{(2^j|x-x'|)^{\rho}, (2^j|x-x'|)^{\rho (n-N)}\}$ and combining \eqref{firstestimate} and \eqref{secondestimate} gives
\[
\int_{R_\rho}\abs{K_j (x,y,z)-K_j (x',y,z)}\,dydz\lesssim c_j(x-x')\quad \forall j\in\na_0.
\]
Now,  let $j_0\geq 0$ be so that $2^{j_0}|x-x'|<1$ and $2^{j_0+1}|x-x'|\geq 1$ and notice that
\[
\sum_{j=0}^\infty c_j(x-x')\leq \sum_{j=0}^{j_0}2^{(j-j_0)\rho}+\sum_{j=j_0+1}^\infty 2^{(j-j_0-1)\rho (n-N)}\lesssim 1,
\]
with the implicit constant independent of $x$ and $x',$ finishing the proof of \eqref{eq:kernelcondxsmall}.

\appendix

\section{Defining $T$ on $L^\infty(\rn)\times L^\infty(\rn)$}\label{app}

The purpose of this appendix is to indicate how an operator $T$ satisfying the hypothesis of Theorem~\ref{thm:boundBMO} can be extended to a bounded operator from $L^\infty(\rn)\times L^\infty(\rn)$ into $BMO$. Our discussion is inspired by the one in Duoandikoetxea \cite[pp. 119-120]{MR1800316}.

Fix $f, g\in L^\infty(\rn).$ Let $Q\subset \rn$ be a cube centered at the origin which contains  $x$ and  let $\qtil$ be the cube containing $Q$ introduced in the proof of Theorem \ref{thm:boundBMO}. We define
\begin{equation}
\label{Qdef}
T_Q(f, g)(x)=T(f\chi_{\qtil}, g\chi_{\qtil})(x)+I_{1,Q}(f, g)(x)+I_{2,Q}(f, g)(x)+I_{3,Q}(f, g)(x),
\end{equation}
where
\begin{align*}
I_{1,Q}(f, g)(x)&=\int_{\rtn}(\mathcal K(x, y, z)-\mathcal K(0, y, z))f(y)\chi_{\qtil}(y)g(z)\chi_{\qtil^c}(z)\,dydz,\\
I_{2,Q}(f, g)(x)&=\int_{\rtn}(\mathcal K(x, y, z)-\mathcal K(0, y, z))f(y)\chi_{\qtil^c}(y)g(z)\chi_{\qtil}(z)\,dydz,\\
I_{3,Q}(f, g)(x)&=\int_{\rtn}(\mathcal K(x, y, z)-\mathcal K(0, y, z))f(y)\chi_{\qtil^c}(y)g(z)\chi_{\qtil^c}(z)\,dydz.
\end{align*}
We note that $T(f\chi_{\qtil}, g\chi_{\qtil})$ is well-defined since $f\chi_\qtil,g\chi_\qtil\in L^\infty_c(\rn)$ and that  the terms $I_{j,Q}, j=1, 2, 3$, are defined through absolutely convergent integrals. Indeed, for  $0<\varepsilon<1$ as in \eqref{eq:kernelcondx}, we have
\begin{align*}
|I_{j,Q}(f, g)(x)|&\lesssim \int_{|x|^\varepsilon\lesssim |x-y|+|x-z|}|\mathcal K(x, y, z)-\mathcal K(0, y, z)|\,dydz\,\,\,\|f\|_{L^\infty}\|g\|_{L^\infty}\\
&\lesssim \|f\|_{L^\infty}\|g\|_{L^\infty}.
\end{align*}

We show next that given cubes $Q, R$ centered at the origin and containing $x$, the definitions of $T_Q(f, g)(x)$ and $T_R(f, g)(x)$ are the same modulo a constant independent of $x$. Without loss of generality, let us assume that $Q\subset R$; in particular, this gives $\qtil\subset \rtil$ as well. Using Remark~\ref{re:Tpointwise}, a straightforward calculation now shows that $T_Q(f, g)(x)-T_R(f, g)(x)$ equals
\begin{align*}
&-T(f\chi_{\rtil\setminus \qtil}, g\chi_{\qtil})-T(f\chi_{\rtil}, g\chi_{\rtil\setminus \qtil})\\
&+\int_{\rtn}(\mathcal K(x, y, z)-\mathcal K(0, y, z))f(y)g(z)\chi_{(\qtil\times\qtil)^c}(y, z)\,dydz\\
&-\int_{\rtn}(\mathcal K(x, y, z)-\mathcal K(0, y, z))f(y)g(z)\chi_{(\rtil\times\rtil)^c}(y, z)\,dydz\\
=&-\int_{\rtn}\mathcal K(x, y, z)f(y)g(z)(\chi_{\rtil\setminus \qtil}(y)\chi_{\qtil}(z)+\chi_{\rtil}(y)\chi_{\rtil\setminus \qtil}(z))\,dydz\\
&+\int_{\rtn}\mathcal K(x, y, z)f(y)g(z)(\chi_{(\qtil\times\qtil)^c}-\chi_{(\rtil\times\rtil)^c})(y, z)\,dydz\\
&-\int_{\rtn}\mathcal K(0, y, z)f(y)g(z)(\chi_{(\qtil\times\qtil)c}-\chi_{(\rtil\times\rtil)^c})(y, z)\,dydz\\
=&-\int_{\rtil\times\rtil\setminus \qtil\times\qtil}\mathcal K(0, y, z)f(y)g(z)\, dydz.
\end{align*}
This gives the desired result since the last integral is independent of $x$ and, because $\mathcal K$ is locally integrable on $\mathbb R^{3n}\setminus \Delta$ and $f,g\in L^\infty(\rn),$ we can assume without loss of generality that it is absolutely convergent. Thus, as a function in $BMO$, we can define   $T(f, g)$ for $f, g\in L^\infty(\rn)$ via the right hand-side of \eqref{Qdef}.

It remains to show that  $\|T(f, g)\|_{BMO}\lesssim \|f\|_{L^\infty}\|g\|_{L^\infty}$. Let $Q\subset \rn$ be an arbitrary cube and let $R\subset \rn$ be some cube centered at the origin such that $Q\subset R$. For $x\in Q$, we can then  write $T(f, g)(x)=T_R (f, g)(x)$. By the proof of Theorem \ref{thm:boundBMO},  since $f\chi_{\rtil}, g\chi_{\rtil}\in L^\infty_c(\rn)$,  there is some constant $C_Q$ such that
$$\dashint_Q|T(f\chi_{\rtil}, g\chi_{\rtil})(x)-C_Q|\,dx\lesssim \|f\|_{L^\infty}\|g\|_{L^\infty},$$
with the implicit constant independent of $Q$. Moreover, it was shown above that $|I_{j,R}(f, g)(x)|\lesssim \|f\|_{L^\infty}\|g\|_{L^\infty}$ for $x\in R$ and $j=1,2,3;$ then we have
\[
\dashint_Q\abs{I_{j,R}(f,g)(x)}\,dx\lesssim \|f\|_{L^\infty}\|g\|_{L^\infty}\quad \text{for } j=1,2,3.
\]
Altogether, we have that  
\[
\dashint_{Q}|T(f, g)(x)-C_Q|\,dx\lesssim \|f\|_{L^\infty}\|g\|_{L^\infty},
\]
with the implicit constant independent of $Q.$ This gives the desired conclusion.



\end{document}